\documentclass[12pt]{article}
\usepackage{amsfonts}
\usepackage{amsmath, amsthm, upref}
\usepackage{enumerate}
\usepackage[parfill]{parskip}
\usepackage{color}
\usepackage{hyperref}
\usepackage{multimedia}
\usepackage{graphics}
\usepackage{epsfig}
\usepackage{graphicx}
\usepackage[greek,english]{babel}
\usepackage{textcomp}
\usepackage{eurosym}

\DeclareMathOperator{\e}{e}

\newcommand{\mbf}{\mathbb{F}}
\newcommand{\mbg}{\mathbb{G}}
\newcommand{\comment}[1]{}

\newcommand{\mcf}{\mathcal{F}}

\newcommand{\gs}{\sigma}

\newtheorem{theorem}{Theorem}

\newtheorem*{corollary}{Corollary}

\newtheorem{example}{Example}

\newtheorem{hypothesis}{Hypothesis}
\newtheorem{lemma}{Lemma}

\newtheorem{remark}[theorem]{Remark}

\newcommand{\bee}{\begin{equation}}
\newcommand{\eee}{\end{equation}}
\newcommand{\bea}{\begin{eqnarray}}
\newcommand{\eea}{\end{eqnarray}}
\newcommand{\bean}{\begin{eqnarray*}}
\newcommand{\eean}{\end{eqnarray*}}

\newcommand{\gep}{\varepsilon}

\begin{document}

\title{ Strict Local Martingales via Filtration Enlargement}
\author{Aditi Dandapani%
\thanks{Applied Mathematics Department, Columbia University, New York, NY 10027; email: ad2259@caa.columbia.edu%
}
\thanks{Supported in part by NSF grant DMS-1308483
} ~and Philip Protter%
\thanks{Statistics Department, Columbia University, New York, NY 10027; email: pep2117@columbia.edu..
}
\thanks{Supported in part by NSF grant DMS-1308483}
}
\date{\today }
\maketitle

\begin{abstract}

A strict local martingale is a local martingale that is not a martingale. We investigate how such a process might arise from a true martingale as a result of an enlargement of the filtration. We study and implement a particular type of enlargement, initial expansion of filtration, for various stochastic differential equations and provide sufficient conditions in each of these cases such that initial expansion can create a strict local martingale.
\end{abstract}

\section{Introduction}\label{intro}

We are interested in mechanisms by which strict local martingales can arise from martingales. A strict local martingale is a local martingale which is not a martingale. We study how expanding the original filtration with respect to which a process is a martingale can lead to a strict local martingale. That is, if we begin with a probability space $(\Omega, \mathcal{F}, \mbf, P)$ where $\mbf$ denotes $(\mcf_t)_{t\geq 0}$, and with an $\mbf$ martingale $M=(M_t)_{t\geq 0}$, and consider an expanded filtration $\mbg$ such that, for all $t$ we have the inclusion $\mathcal{F}_t \subset \mathcal{G}_t,$ when can we obtain a filtration $\mbg$ such that $M$ becomes a strict local martingale, possibly under a different but equivalent probability measure $Q$? 

At first sight it might seem like a strange construction, to enlarge a filtration and change the probability measure. We will argue that it is a natural thing to do from the standpoint of Mathematical Finance.

Strict local martingales have recently been a popular subject of study. Some relatively recent papers concerning strict local martingales include Biagini et al~\cite{BFN}, Bilina-Protter~\cite{Roseline}, Chybiryakov~\cite{OC}, Cox-Hobson~\cite{CH}, Delbaen-Schachermayer~\cite{DS}, F\"ollmer-Protter~\cite{FP},  Lions-Musiela~\cite{L-M}, Hulley~\cite{HH}, Keller-Ressel~\cite{K-R}, Klebaner-Liptser~\cite{KL}, Kreher-Nikeghbali~\cite{KN}, Larsson~\cite{Martin}, Madan-Yor~\cite{MY}, Mijatovic-Urusov~\cite{MU}, Protter~\cite{PhilipToBe}, Protter-Shimbo~\cite{PhilipKazu}, and Sin~\cite{C-S}, and from this list we can infer a certain interest. Our motivation comes from the analysis of financial bubbles, as explained in~\cite{PhilipToBe}, for example. The theory tells us that on a compact time set, the (nonnegative) price process of a risky asset is in a bubble, i.e., undergoing speculative pricing, if and only if the price process is a strict local martingale under the risk neutral measure governing the situation. Therefore one can model the formation of bubbles by observing when the price process changes from being a martingale to being a strict local martingale. This is discussed in detail in~\cite{JPS2}, \cite{BFN}, and~\cite{PhilipToBe}, for example. 

The models we study are stochastic volatility models. We work with the setting examined in Lions and Musiela ~\cite{L-M}, which provides necessary and sufficient conditions such that the solutions of such stochastic differential equations are strict local martingales. We assume always that a component of the stochastic volatility process is an It\^{o} diffusion, so that we can  use Feller's test for explosions in our quest to characterize the stochastic processes in question. This is similar to techniques used in~\cite{BFN} and~\cite{C-S},  but with the difference that we introduce a cause for bubbles (new information available to the market), and then show how this mathematically evolves into a bubble. 

The expansion of  filtration using initial expansion involves adding the information encoded in a random variable to the original $\gs$ algebra at time zero. It then propagates throughout the filtration. This augmentation doesn't have to happen at time zero, however; it can happen at any finite valued stopping time $\tau$. This is due to the fact that at $\tau$ we know what is happening, and thus we can think of an enlargement beginning at $\tau$ exactly analogously to one beginning at time $t=0$, with simply the time $\tau$ playing the role of the time $t=0$. From now on however, we will deal with enlargements at time $t=0$, for notational simplicity. 

This type of enlargement of filtration from $\mbf$ to $\mbg$ changes the semimartingale decomposition of the underlying price process, and therefore leads to a change of a risk-neutral measure from $P$ to an equivalent probability measure $Q$. Our stochastic process, which we will call $S,$ which is assumed to be a martingale under $(P,\mbf)$, under certain conditions can become a strict local martingale on a stochastic interval that depends on the choice of $Q$ and the random variable that we add to $\mbf.$ This random variable is denoted $L$.

The case of initial expansions is particularly tractable, since Jacod~\cite{JAC} has developed the theory that provides us with the dynamics of the process under the enlarged filtration. That is, he provides us with the semimartingale decomposition of the process in the enlarged filtration, which permits us to choose a risk neutral measure $Q$ for the enlarged filtration that removed the enlargement created drift. Under that $Q$ we can sometimes detect the presence of the strict local martingale property of the process, or lack thereof. 

An outline of our paper is as follows. After an introduction, Section~\ref{s2} is the heart of the paper. Here we present the model of P.L. Lions and M. Musiela on stochastic volatility (in the style of what are known as Heston-type models), and we show how the addition of more information via an ``expansion of the filtration'' can lead what was originally a martingale to become a strict local martingale, under a risk neutral measure chosen from the infinite selection available in an incomplete market.  Our main results are Theorem~\ref{t1} and Theorem~\ref{t2}. In Section~\ref{s3} we drop the hypothesis of continuous paths and extend our results to the case of discontinuous martingales replacing Brownian motions. Our main result in this section is Theorem~\ref{t7}. In Section~\ref{s4} we conclude with a description of how these results relate to mathematical finance.

\section{The Framework of Lions and Musiela}\label{s2}

\subsection{Our First Model}\label{first}

Let us begin with the framework established by P.L. Lions and M. Musiela~\cite{L-M}, that treats the case of stochastic volatility. We will begin working on a probability space $(\Omega, \mathcal{F}, \mbf, P),$ where $\mbf=(\mathcal{F}_t)_{0\leq t \leq T}$. We assume that  the stochastic process $S=(S_t)_{0\leq t\leq T}$, which we can think of as a stock price, and the stochastic volatility satisfy SDEs of the following system of two equations:

 \begin{eqnarray} dS_t&=&S_t v_tdB_t; \qquad S_0=1\\
  dv_t&=&\mu(v_t)dW_t +b(v_t)dt;  \quad v_0=1.
 \end{eqnarray}
 
Here $B$ and $W$ are correlated Brownian motions, with correlation coefficient $\rho$. Our time interval is assumed to be $[0,T].$ 
We will assume that $\mu$ and $b$ are $C^1$ functions on $[0,\infty)$ and that $\mu$ is Lipschitz continuous on $[0,\infty)$ such that: 

\begin{eqnarray*}
\mu(0)&=&0\\
b(0) &\geq& 0\\
 \mu(x) \textgreater 0&\text{ if }&x>0\\ 
 b(x)&\leq& C(1+x)
 \end{eqnarray*} 

We recall the conditions of Lions and Musiela, which allow us to determine whether the solution to $(1)$ is a strict local martingale or an integrable, non-negative martingale: If

           \begin{equation*}  \limsup_{x \to +\infty} \frac{\rho\ x\mu(x)+b(x)}{x} \textless \infty\  \end{equation*} 
           holds, then $S$ is a non negative martingale.           
          
             For the same model, if the condition
                
                \begin{equation*}   \liminf_{x \to +\infty}(\rho\ x\mu(x)+b(x) )\phi(x)^{-1} \textgreater 0\  \end{equation*} 
                
                holds, then $S$ is not a martingale but a supermartingale and a strict local martingale.
Here, $\phi(x)$ is an increasing, positive, smooth function that satisfies 
\begin{equation}   \int_{a}^{\infty} \frac{1}{\phi(x)}dx \textless \infty \end{equation} 
with $a$ being some positive constant.

\begin{remark}
Lions and Musiela~\cite{L-M} assume that $\mu$ and $b$ are both $C^\infty$, but if one reads their proof they do not use the force of that assumption. Assuming they are $C^1$ as we have done is enough for their proofs to work.
\end{remark}                
                
 We would like to determine whether or not an enlargement of the filtration can give rise to a strict local martingale in the bigger filtration, when one begins with a true martingale in the smaller one. More specifically, we would like to answer the following question: beginning with a probability space $(\Omega, \mathcal{F}, \mbf, P)$, and a price process $S$ that is an $\mbf$ martingale, if we perform a countable expansion of $\mbf$, resulting in an enlarged filtration $\mbg$, can we obtain a $\mbg$ strict local martingale under an equivalent measure?

\subsection{The Case of Initial Expansions}\label{Init}

We will consider the case of initial expansions, i.e. the expansion of the filtration $\mbf$ by adding a random variable $L \in \mathcal{F}$ to $\mathcal{F}_0.$ We assume that this random variable $L$ takes values in a Polish space $(E, \mathcal{E}).$ The new, enlarged filtration, which we will call $\mbg$ can be denoted as $$\mathcal{G}_t=\bigcup_{\epsilon \textgreater 0}(\mathcal{F}_{t+\epsilon} \vee \sigma(L))$$

We use the results of Jean Jacod~\cite{JAC} on the initial expansion of filtrations: If $S$ is a continuous $\mbf$ martingale, there exists a process $(x,\omega,t) \rightarrow k^{x}(t,\omega),$ measurable with respect to the sigma algebra $\mathcal{E} \otimes \mathcal{P}(\mathbb{F})$, where ${P}(\mathbb{F})$ denotes the predictable sigma algebra on $\Omega \otimes \mathbb{R}_{+}$, such that $\langle q^{x},S \rangle =(k^{x}q^{x}_{-}) \cdot \langle S,S\rangle$. 

The function $q^{x}$ is given by the following: Let $\eta$ be the distribution of $L$, and let $Q_t(\omega, dx)$ be the regular conditional distribution of $L$, given $\mathcal{F}_t.$ The process $q^{x}(t,\omega)\eta(dx)$ is an $\mbf$ martingale, and a version of $Q_t(\omega,dx).$    

We have 
\begin{equation*} Q_t(w, \cdot)=q_t(w, \cdot)\eta(\cdot) \end{equation*}
Jacod proves the existence of an $\mbf$ predictable process $(k^{L}_t)_{0\leq t\leq T}$ such that in our case, where we are assuming $S$ has continuous paths,\footnote{Note that when $S$ has continuous paths the process $\langle S,S\rangle = [S,S]$, and thus here we write it in the simpler form $[S,S]$. See for example~\cite{PP} for discussion of the relations between the processes $\langle\cdot,\cdot\rangle$ and $[\cdot,\cdot]$.}
 \begin{equation*} [q^{L},S]=k^{L}q^{L}_{-} \cdot [S,S] \end{equation*}
 We have used the left continuous version of $q^L$ in the above equation to ensure it is predictable, but since $[S,S]$ is continuous here, we could just as well have used the process $q^L$ itself, without bothering about a left continuous version.
The process $(k^{L}_t)_{t\geq 0}$ satisfies $k^{L}=\frac{h^{L}}{q^{L}}$ if $q^{L} \textgreater 0$ and $k^{L}_t=0$ otherwise. In the above, $h^{L}_s$ is the density process such that we have \begin{equation} d[q^{L},S]_s=h^{L}_sd[S,S]_s \end{equation}
Jacod's theorem tells us next that the following process is a $\mbg$ local martingale: 
\begin{equation} \widetilde{S_t}=S_t-\int_{0}^{t}k^{L}_sd[S,S]_s \end{equation} 
By saying it is a $\mbg$ local martingale, we are not precluding that it is a martingale; we need to have extra conditions to conclude it is a strict local martingale. 

Let us illustrate this concept with an example, the case where the random variable $L$ takes on only a finite number of values:\\ Let $A_1, A_2....A_n$ be a sequence of events such that  $A_i \cap A_j= \emptyset$ if $ i \neq j$ and $\bigcup\limits_{i=1}^n A_k=\Omega.$ The enlarged filtration, $\mbg$, is the filtration generated by $\mbf$ and the random variable $L= \sum\limits_{i=1}^n a_i1_{A_i}.$  \\ 

In this case, we have \begin{equation*} k^{L}_tq^{L}_t= \sum\limits_{i=1}^n \frac{\xi^{i}_t}{P(A_i)}1_{\{L=a_i\}},\end{equation*} and \begin{equation*} q^{L}_t= \sum\limits_{i=1}^n \frac{P(L=a_i |\mathcal{F}_t)}{P(A_i)}1_{\{L=a_i\}} \end{equation*}

Here, $\xi^{i}_t$ are processes arising from the Kunita-Watanabe inequality, which ensures absolute continuity of the paths: If we let $N^{i}_t$ be the $\mathcal{F}$-martingale $P(L=a_i |\mathcal{F}_t)$ we have that $d[N^{i},S]_t=\xi^{i}_td[S,S]_t.$

We will henceforth work with the general case of initial expansions, wherein we don't necessarily have a countable partition of the sample space.

Returning to the Lions-Musiela framework, we have that, under $(P, \mbg),$ the price process and stochastic volatility satisfy: 

$$S_t=\int_{0}^{t}(S_sv_s)dB_s- \int_{0}^{t}k^{L}_sd[S,S]_s+ \int_{0}^{t}k^{L}_sd[S,S]_s,$$

where 
$$
\int_{0}^{t}(S_sv_s)dB_s- \int_{0}^{t}k^{L}_sd[S,S]_s
$$ 
is a $(P,\mbg)$ martingale, and 
$$
\int_{0}^{t}k^{L}_sd[S,S]_s
$$ 
is a finite variation process. The stochastic volatility in turn satisfies: $$ v_t=v_0+\int_{0}^{t}\mu(v_s)dW_s-\int_{0}^{t}k^{L}_s\mu(v_s)^{2}ds +\int_{0}^{t}k^{L}_s\mu(v_s)^{2}ds +\int_{0}^{t}b(v_s)ds.$$ Here, $$v_t=v_0+\int_{0}^{t}\mu(v_s)dW_s-\int_{0}^{t}k^{L}_s\mu(v_s)^{2}ds$$ is a $(P,\mbg)$ martingale, and $$\int_{0}^{t}k^{L}_s\mu(v_s)^{2}ds +\int_{0}^{t}b(v_s)ds$$ is a finite variation process. 

We perform a Girsanov transform, to switch to a probability measure under which $S$ is a local martingale: under $(Q, \mbg)$, where the measure $Q$ is equivalent to $P,$ $S$ possesses the following decomposition: $$S_t=\int_{0}^{t}(S_sv_s)dB_s- \int_{0}^{t}k^{L}_s(S_sv_s)^{2}ds-\int_{0}^{t}\frac{1}{Z}d[Z,\sigma \cdot B]_s +\int_{0}^{t}k^{L}_s(S_sv_s)^{2}ds+ \int_{0}^{t}\frac{1}{Z}d[Z,\sigma \cdot B]_s.$$ Here, $Z_t=E[\frac{dQ}{dP}|\mathcal{G}_t].$ Writing 
$$
Z_t=1+ZH\cdot B_t +ZJ \cdot W_t,
$$ 
where $\cdot$ represents stochastic integration, for predictable processes $J$ and $H$, we have the above equal to (recall that $d[B,W]_t=\rho dt$):

 $$ \int_{0}^{t}(S_sv_s)dB_s- \int_{0}^{t}k^{L}_s(S_sv_s)^{2}ds-\int_{0}^{t}((S_sv_s)H_s+ \rho J_s)ds+ \int_{0}^{t}k^{L}_s (S_sv_s)^{2}ds +\int_{0}^{t}((S_sv_s)H_s+\rho J_s)ds.$$ In order to get rid of the finite variation term in this decomposition, we set $$ \int_{0}^{t}k^{L}_s(S_sv_s)^{2}ds +\int_{0}^{t}((S_sv_s)H_s+\rho J_s)ds=0.$$
 
 The volatility, in turn, has the following decomposition: \begin{equation*}v_t=\int_{0}^{t}\mu(v_s)dW_s-\int_{0}^{t}k^{L}_s\mu(v_s)^{2}ds-\int_{0}^{t}\frac{1}{Z}d[Z,\mu\cdot W]_s+\int_{0}^{t}k^{L}_s\mu(v_s)^{2}+\int_{0}^{t}b(v_s)ds +\int_{0}^{t}\frac{1}{Z}d[Z,\mu\cdot W]_s \end{equation*}

 This equals: 
 
  \begin{multline*}  v_t= \int_{0}^{t}\mu(v_s)dW_s-\int_{0}^{t}k^{L}_s\mu(v_s)^{2}ds-\int_{0}^{t}(\rho \mu(v_s)H_s+\mu(v_s)J_s)ds+\int_{0}^{t}b(v_s)ds+ \int_{0}^{t}k^{L}_s\mu(v_s)^{2}ds+\\ \int_{0}^{t}(\rho \mu(v_s)H_s+\mu(v_s)J_s)ds
  \end{multline*}

Here $$\int_{0}^{t}\mu(v_s)dW_s-\int_{0}^{t}k^{L}_s\mu(v_s)^{2}ds-\int_{0}^{t}(\rho \mu(v_s)H_s+\mu(v_s)J_s)ds$$ is a $(Q, \mbg)$ local martingale, and $$ \int_{0}^{t}b(v_s)ds+ \int_{0}^{t}k^{L}_s\mu(v_s)^{2}ds+\int_{0}^{t}(\rho \mu(v_s)H_s+\mu(v_s)J_s)ds$$ is a finite variation process.
  
 Recall that under $(Q, \mbg)$, we would like $S$ to be a local martingale. This entails the finite variation part of the decomposition of $S$ under $(Q, \mbg)$ being zero. Given that there are two Brownian motions, there are infinitely many combinations of $H$ and $J$ that will work. What we need is 
\bee\label{nne1}
k^{L}_t(S_tv_t)^{2}= -(S_tv_t)H_t -\rho J_t.
\eee
Note that $k^L_t$ in our framework is defined by the relation, for a right-continuous martingale $(q^{L}_t)_{t\geq 0}$, by 
$$
[q^L,S]_t=\int_0^tk^L_sq^L_sd[S,S]_s=\int_0^tk^L_sq^L_s(S_sv_s)^2ds
$$

In the rest of this paper we will make the following assumptions on the processes $k$, $H$ and $J:$ 

\bigskip

\begin{hypothesis}[Standing Assumptions]\label{SA}
\begin{eqnarray}
\label{SA1}&&\text{We} \text{ assume that each of }k,H,\text{ and }J\text{ have right continuous paths }a.s.\\
\label{SA2}&&Q(\omega: k^{L}_0 \textgreater 0) \textgreater 0\\
\label{SA3}&&Q\text{ is a true probability measure}
\end{eqnarray}
\end{hypothesis}

We note that in Subsection~\ref{ss1} we give examples and also a framework where the important process $k^L$ has right continuous paths, a.s., which shows that Hypothesis (1) is not unreasonable.  Our new drift, which we will call $\hat{b}(v_t),$ is given by 
\bee\label{nne4}
\hat{b}(v_t)=b(v_t)+k^{L}_t\mu^{2}(v_t)+(\rho H_t+J_t)\mu(v_t).
\eee

Notice that we can no longer represent the drift in deterministic terms as simply functions of the real variable $x$, so we cannot immediately invoke the results of Lions \& Musiela. To address this, let us fix $0 \textless \gep^{(1)} \textless k^{L}_0$ and $ |\rho H_0+J_0| \textless \gep^{(2)}$ and define the following random times: 
  
 \begin{eqnarray*}
\tau^{k}&=& \inf \{t: |k^{L}_t| \textless \gep^{(1)}\}\\   
\tau^{H,J}&=&\inf \{t: |\rho H_t+J_t| \textgreater \gep^{(2)}\}
\end{eqnarray*} 

Note that since the processes $k^{L}_t,$ $H_t$ and $J_t$ are assumed to be $\mbg$ predictable, right continuous processes, we can indeed claim that these random times are $\mbg$ stopping times, by the theory of d\'ebuts, as originally developed by Dellacherie ~\cite{Del}.

Now define the stopping time $\tau$ to be 
  \begin{equation}\label{tau}
  \tau=(\tau^{k} \wedge \tau^{H,J}).
  \end{equation}
 By ~\eqref{SA2} we have that $Q(\tau \textgreater 0) \textgreater 0.$
    
On the stochastic interval $[0,\tau],$ we have the following lower bound on our drift coefficients: 
\begin{equation*}
\hat{b}(v_t)= b(v_t)+k^{L}_t\mu^{2}(v_t)+(\rho H_t+J_t)\mu(v_t) \geq b(v_t)+\min(\gep^{(1)}, \gep^{(2)})\mu^{2}(v_t)-\max(\gep^{(1)}, \gep^{(2)})\mu(v_t)
\end{equation*}

 Before we state the next result, we state and prove a technical lemma.
 
 \medskip
 
 \begin{lemma}\label{ell1}
 Let $B$ be a standard Brownian motion, and let $\nu$ be any continuous, adapted, finite valued process such that $\int_0^t\nu_s^2ds<\infty$ a.s. for each $t>0$. Suppose $\gs$ is continuous and is such that  $S$ exists and is the unique solution of 
\bee
S_t=1+\int_0^tS_s\gs(S_s)\nu_sdB_s
\eee
Then $S$ is strictly positive for all $t\geq 0$ a.s.
\end{lemma}
\begin{proof}
We let $V=\inf\{t>0: S_t=0\}$. It suffices to show that $P(V<\infty)=0$.

Define stopping times $R_n=\inf\{t>0: S_t=1/n\text{ or }S_t=S_0\vee n\}$. 
Note that $P(R_n>0)=1$ for $n\geq 2$ because $S_0=1$ a.s. and $S$ is continuous. Use It\^o's formula up to time $R_n$ to get
\bee\label{eph1}
\ln(\vert S_{R_n}\vert)=\ln(S_0)+\int_0^{R_n}S_s\gs(S_s)\nu_sdB_s-\frac{1}{2}\int_0^{R_n}S_s^2\gs(S_s)^2\nu_s^2ds
\eee
The stopping times $R_n$ increase to $V$ as $n$ tends to $\infty$, so the left side of~\eqref{eph1} tends to $\infty$ on the event $\{V<\infty\}$. The right side however remains finite ($\gs(x)$ is assumed continuous, and is therefore bounded on compact sets) on $\{V<\infty\}$, and the only way this can happen is if $P(V<\infty)=0$.

\end{proof}

 \begin{remark} The work of Engelbert and Schmidt~\cite{ES} gives necessary and sufficient conditions for a solution to exist that is unique in law, at least when $\nu$ is not present. Lemma~\ref{ell1} remains true under more general hypotheses, with the obvious modifications of the proof. As such it is a slight extension of~\cite[Theorem 71 of Chapter V]{PP}.
 \end{remark}

 The above discussion gives us the following result.

 \begin{theorem}\label{t1}
Assume $\mu$ and $b$ are $\mathcal{C}^1$ and that Hypothesis~\ref{SA} holds, as well as the following conditions: 
\begin{eqnarray*}  
& \limsup_{x \to +\infty}& \frac{\rho\ x\mu(x)+b(x)}{x} \textless \infty\\
& \liminf_{x \to +\infty}&(\rho\ x\mu(x)+b(x) +\min(\gep^{(1)}, \gep^{(2)})\mu^{2}(x)-\max(\gep^{(1)}, \gep^{(2)})\mu(x))\phi(x)^{-1} \textgreater 0
\end{eqnarray*}
 on the functions $\mu$,$b$ are satisfied, and assume that $B$ and $W$ are correlated Brownian motions with correlation $\rho>0$. Let the process $S$ be the unique strong solution of the SDE 
 
\begin{eqnarray*} 
dS_t&=&S_t v_tdB_t\\
dv_t&=&\mu(v_t)dW_t +b(v_t)dt  
\end{eqnarray*}
    
on $( P,\mbf)$. The solution $S$ is also the solution of
\begin{eqnarray*} 
dS_t&=&S_t v_tdB_t\\
dv_t&=&\mu(v_t)dW_t +b(v_t)dt  +k^{L}_t\mu^{2}(v_t)dt+ (\rho H_t+J_t)\mu(v_t)dt
 \end{eqnarray*} 
on $(Q, \mbg)$. Then $S$ is a  positive $( P,\mbf)$ martingale and a positive $(Q, \mbg)$ \textit{strict} local martingale on the stochastic interval $[0,\tau]$, where $\tau$ is given in~\eqref{tau}. More specifically, we have $E[S^{\tau}_t] \textless S_0.$
 
 In the above, $\phi(x)$ is an increasing, positive, smooth function that satisfies 
 \begin{equation*}  \int_{a}^{\infty} \frac{1}{\phi(x)}dx \textless \infty \end{equation*} 

\end{theorem}

\begin{remark} The condition $\rho>0$ assumed in Theorem 3 is used in the proof of the quoted result of Lions and Musiela, which is why we need to assume it.
\end{remark}
    
Before we continue, let us recall a result (proved for example in~\cite{PP}) that allows us to compare the values of solutions of stochastic differential equations. It is well known, but we include it here for the reader's convenience. Let us denote by $\mathcal{D}^{n}$ the set of $\mathbb{R}^{n}-$ valued c\`adl\`ag processes. We write $\mathcal{D}$ for $\mathcal{D}^1$. An operator $\bf{F}$ from $\mathcal{D}^{n}$ to $\mathcal{D}$ is said to be \textit{Process Lipschitz} if for all $X,$ $Y \in \mathcal{D}^{n}$ and for all stopping times $T:$
 
 \begin{enumerate}
 
 \item $X^{T_{-}}=Y^{T_{-}} \Rightarrow F(X)^{T_{-}}=F(Y)^{T_{-}}$
 \item There exists an adapted process $K_t$ such that $||F(X)_t-F(Y)_t|| \leq K_t ||X_t-Y_t||$, where $\parallel\cdot\parallel$ denotes the sup norm. 
 \end{enumerate}
 
 \begin{theorem}[Comparison Theorem]\label{comp}~\cite[p. 324]{PP}
 Let $Z$ be a continuous semimartingale, let $F$ be process Lipschitz, and let $A_t$ be adapted, increasing, and continuous. Assume that $G$ and $H$ are process Lipschitz functionals such that $G(X)_{t_{-}} \textgreater H(X)_{t_{-}}$ for all $t>0$ and all semimartingales $X.$ Let $x_0\geq y_0$, and $X$ and $Y$ be the unique solutions of 
 
 \begin{eqnarray*} 
  X_t=x_0+\int_{0}^{t}G(X)_{s_{-}}dA_s+\int_{0}^{t}F(X)_{s{-}}dZ_s\\
 Y_t=y_0+\int_{0}^{t}H(Y)_{s_{-}}dA_s+\int_{0}^{t}F(Y)_{s{-}}dZ_s
 \end{eqnarray*}
 
 Then, $P\{ \exists  t \geq 0: X_t \leq Y_t\} =0.$
  \end{theorem}
 Now we may begin the proof of Theorem~\ref{t1}.
 
 \begin{proof}[Proof of Theorem~\ref{t1}]
That $S$ is positive follows from Lemma~\ref{ell1}. Let us begin using the framework $(P,(\mathcal{F}_t)_{0\leq t\leq T})$. Notice that the condition  
\begin{equation*} \limsup_{x \to +\infty} \frac{\rho\ x\mu(x)+b(x)}{x} \textless \infty\  
\end{equation*} 
is sufficient to show that the solution to the SDE 
 \begin{eqnarray} dS_t&=&S_t v_tdB_t\\
 dv_t&=&\mu(v_t)dW_t +b(v_t)dt  
 \end{eqnarray}
is a true martingale (Lions and Musiela~\cite[Theorem 2.4(i)]{L-M}) and that the condition  
\begin{equation*} 
\liminf_{x \to +\infty}(\rho\ x\mu(x)+b(x) +\epsilon \mu^{2}(x)- \epsilon \mu(x))\phi(x)^{-1} \textgreater 0\ 
\end{equation*} 
is enough to show that the solution to the SDE  
\begin{eqnarray} dS_t&=&S_t v_tdB_t\\
dv_t&=&\mu(v_t)dW_t +b(v_t)dt + \epsilon \mu^{2}(v_t)dt -\epsilon \mu(v_t)dt
\end{eqnarray}
is a strict local martingale, by Lions and Musiela~\cite[Theorem 2.4 , (ii)]{L-M}. Note that $S$ is a nonnegative supermartingale, so its expectation is non-increasing with time.
    
Define a sequence of stopping times $T_n$ by $\inf \{t: v_t \geq n\},$. We have that the stopped process $S_{t \wedge \tau \wedge T_n}$ is a martingale.  The stopping time $T_\infty$ is the explosion time of $v.$ Therefore, we may write $$S_0=E[S_{t \wedge \tau \wedge T_n}]= E[S_{t \wedge \tau} 1_{\{t \wedge \tau \textless T_n\}}]+ E[S_{T_n} 1_{\{T_n \leq t \wedge \tau\}}].$$  Since $E[S_{t \wedge \tau} 1_{\{t \wedge \tau \textless T_n\}}]$ converges to $E[S_{t \wedge\tau}]$, we would have that $E[S_{t \wedge \tau}]  \textless S_0$ for all $t$ if we can show that  $\liminf_{n \to +\infty} E[S_{T_n} 1_{\{T_n \leq t \wedge \tau\}}] \textgreater 0.$ 
   
   We have: $ E[S_{T_n} 1_{\{T_n \leq t \wedge \tau\}}] =\hat{P}(T_n \leq t \wedge \tau)$ where under $\hat{P},$ $v$ solves $$dv_t=\mu(v_t)dW_t+b(v_t)dt+ \rho \mu(v_t)v_tdt+ k^{L}_t\mu^{2}(v_t)+(\rho H_t+J_t)\mu(v_t)dt$$
   
Now, the condition   
 \begin{equation*} \liminf_{x \to +\infty}(\rho\ x\mu(x)+b(x)+ \min(\gep^{(1)}, \gep^{(2)}) \mu^{2}(x)-\max(\gep^{(1)}, \gep^{(2)})\mu(x))\phi(x)^{-1} \textgreater 0\ \end{equation*} 
 is sufficient to guarantee that the explosion time of the stochastic differential equation  
\begin{equation*}
  dv_t=\mu(v_t)dW_t +b(v_t)dt + \rho v_t \mu(v_t)+ \min(\gep^{(1)}, \gep^{(2)}) \mu^{2}(v_t)dt-\max(\gep^{(1)}, \gep^{(2)})\mu(v_t)dt
    \end{equation*}
 can be made as small as we wish in probability under the measure $P$, an hence as well under the measure $\hat{P}$, since $\hat{P}$ is absolutely continuous with respect to $P$, and convergence in probability under $P$ therefore implies convergence in probability under $\hat{P}$.
 
 It is easy to see that the comparison theorem stated above implies that the solution to the SDE $$dv_t=\mu(v_t)dW_t+b(v_t)dt+ \rho \mu(v_t)v_tdt+ k^{L}_t\mu^{2}(v_t)+(\rho H_t+J_t)\mu(v_t)dt$$ is $Q$ almost surely greater than or equal to that of the SDE  $$dv_t=\mu(v_t)dW_t +b(v_t)dt +\min(\gep^{(1)}, \gep^{(2)}) \mu^{2}(v_t)dt-\max(\gep^{(1)}, \gep^{(2)}) \mu(v_t)dt$$ for all $t \in [0, \tau].$ Thus since the explosion time of $v$ in the SDE  
\begin{eqnarray*} 
dS_t&=&S_t v_tdB_t\\
dv_t&=&\mu(v_t)dW_t +b(v_t)dt + \min(\gep^{(1)}, \gep^{(2)}) \mu^{2}(v_t)dt-\max(\gep^{(1)}, \gep^{(2)})\mu(v_t)dt 
\end{eqnarray*}
can be made as small as possible, the explosion time $T_\infty$ of $v$ in the SDE
\begin{eqnarray*}
dS_t&=&S_t v_tdB_t\\
dv_t&=&\mu(v_t)dW_t+b(v_t)dt+ \rho \mu(v_t)v_tdt+ k^{L}_t\mu^{2}(v_t)+(\rho H_t+J_t)\mu(v_t)dt\end{eqnarray*}
 can be made as small as possible as well. 
 
 This means that, for all $t$, we have $\hat{P}(T_\infty \leq t \wedge \tau) \textgreater 0$.  This implies that for all $t$ we have $$E[S^{\tau}_{t}]=\textless S_0,$$ implying that $S_t$ is a local martingale that is not a martingale, and hence a strict local martingale.     
 \end{proof}

 \begin{remark}
 It can be checked that the functions $\mu(x)=x$ and $b(x)=x-\rho x^{2}$ satisfy the criteria  
 
  \begin{equation*}  \limsup_{x \to +\infty} \frac{\rho\ x\mu(x)+b(x)}{x} \textless \infty\  \end{equation*} 

 \begin{equation*} \liminf_{x \to +\infty}(\rho\ x\mu(x)+b(x) +\min(\gep^{(1)}, \gep^{(2)})\mu^{2}(x)-\max(\gep^{(1)}, \gep^{(2)})\phi(x)^{-1} \textgreater 0\ \end{equation*} 
 
 In fact, for $k \geq 1$, the functions  $\mu(x)=x^{k}$ and $b(x)=x-\rho x^{k+1}$ work as well, \textit{if $\rho$ is positive}. The reason we need $\rho$ to be positive is that we need the following condition on the drift, in order for it to have a non-exploding, positive solution: $$b(0) \geq 0$$ $$b(x) \leq C(1+x)$$ for some $C \geq 0.$
 
 Thus, if we work with an SDE with such diffusion and drift coefficients, we begin with a true martingale and end up with a strict local martingale, due to initial expansions.
 
 \end{remark}
 
 Indeed, one can check using Feller's test for explosions (see, for example, ~\cite{KS}), that the SDE
$$
dv_t=v^{k}_tdW_t+(v_t-\rho v^{k+1}_t)dt
$$ 
does not explode (in other words, that the time of explosion is infinite, almost surely): If we assume our state space for $v$ to be $(-\infty, +\infty)$, we need only to show that the scale function, $$p(x)=\int_{c}^{x}e^{\{-2\int_{c}^{\psi}\frac{b(y)}{\mu^{2}(y)}dy\}}d\psi$$ satisfies the following: $$p(-\infty)=-\infty$$ $$p(\infty)=\infty$$ In the above expression for the scale function, $\mu(x)$ is the diffusion coefficient, and $b(x)$ is the drift. If we take $\mu(x)=x$, and $b(x)=x-\rho x^{2}$, a quick computation shows that indeed $$p(-\infty)=-\infty$$ and $$p(\infty)=\infty.$$

 \begin{remark} One should note that in the case that the random variable $L$ is independent of the sigma algebra generated by by the process $S=(S_t)_{0\leq t\leq T}$, we have the process $k^{L}$ identically equal to zero. This is because the decomposition of the process $S$ does not change under an expansion of filtrations, and the martingale nature of solutions doesn't change.   
 \end{remark} 

Before we continue, we must ensure that the subprobability measure $Q$ defined above is a \textit{true} probability measure. Let us begin by defining the sequence of probability measures $Q_m$ by $$dQ_m=Z_{T\wedge T_m}dP$$ where $T_m=  \inf\{t: \int_{0}^{t} (H^{2}_s+J^{2}_s+2\rho J_sH_s)ds \ge h(m)\},$ for some function $h$. We then have
  
$$E[e^{\frac{1}{2}\int_{0}^{t \wedge T_m}(H^{2}_s+J^{2}_s+2\rho J_sH_s)ds}] \leq e^{\frac{1}{2}h(m)} \textless \infty.$$ Recall that the relation $$k^{L}_t(S_tv_t)^{2}= -(S_tv_t)H_t -\rho J_t$$ holds true for all $t \geq 0.$

So we have $Q_m \ll P$ on $[0,T_m]$ for each $m,$ as well as that the $Q_m$ are true probability measures, since $Z_{t}^{T_m}$ is a true $\mathbb{G}$ martingale. 
  
Note that if $\{Z_{T \wedge T_m}\}_m$ is a uniformly integrable martingale, then $Q$ is equivalent to $P$ on $[0,T].$ This is because the uniform integrability of $(Z^{T_{m}})_{m}$ ensures the $L^{1}$ convergence of $Z^{T_m}$, i.e. $$\lim_{m \to \infty}E[Z_{t \wedge T_m}]=E[Z_{t \wedge \widetilde{T}}]=1,$$  where $\widetilde{T}=\lim_{m \to \infty}T_m.$ It is assumed that $\widetilde{T} \geq T.$ So we obtain that, for all $t$ in the interval $[0,T]:$ $E[Z_{t \wedge \widetilde{T}}]=E[Z_t]=1.$ Thus, $Q$ is equivalent to $P$ on $[0,T].$ 
  
\subsection{A Slightly More General Model}\label{alpha}
We next perform a similar analysis for the following case: 
\begin{eqnarray} 
dS_t&=&S_t^{\beta} v_t^{\delta}dB_t\\
  dv_t&=&\alpha v_t^{\gamma}dW_t +b(v_t)dt  
    \end{eqnarray}
  Here, we make the following assumptions and restrictions on the parameters and functions: $\alpha$, $\gamma$, $\beta$, and $\delta$ are all positive, $b(0) \geq 0$, $b$ is Lipschitz on $[0, \infty)$ and satisfies, for all $x,$ $$b(x) \leq C(1+x)$$ In addition, we assume that $$\mu(0)=0,$$ $$\mu(x) \textgreater 0, x \textgreater 0,$$ and lastly that $\mu$ is locally Lipschitz.

Next we note that if $\beta \textless 1$, we have that the process $S$ is a true martingale possessing moments of all orders. The interesting case is when $\beta \geq 1,$ and assume no further restrictions on $\gamma,$ since with the conditions specified on $b,$ the above system of stochastic differential equations will not explode.  The details of this case are almost identical to that of the previous case, and we omit most of them here. We work with a new probability measure and enlarged filtration $(Q, \mbg).$ The measure $Q$ is defined by $E[\frac{dQ}{dP}]\mathcal{G}_t]=Z_t$. We choose $Q$ such that it is a local martingale measure for $S.$
Writing 
$$
Z_t=1+ZH\cdot B_t +ZJ \cdot W_t,
$$ 
where $\cdot$ represents stochastic integration, for predictable processes $J$ and $H$, and recalling that $d[B,W]_t=\rho dt$, we arrive at the $(Q, \mbg)$ decomposition for the volatility after doing a calculation very similar to that done for the previous model:
\begin{eqnarray*}
v_t&=&v_0+\int_{0}^{t}\alpha v^{\gamma}_sdW_s-\int_{0}^{t}\alpha^{2}k^{L}_sv^{s\gamma}ds-\int_{0}^{t}(\alpha v^{\gamma}_sH_s\rho+\alpha J_sv^{\gamma}_s)ds+ \int_{0}^{t}b(v_s)ds\\
&& +\int_{0}^{t}\alpha^{2}k^{L}_sv^{2\gamma}_sds+ \int_{0}^{t}(\alpha v^{\gamma}_sH_s\rho+\alpha J_sv^{\gamma}_s)ds.
\end{eqnarray*} 
  
 Our new drift derivative, then, call it $\hat{b}(v_t)$ satisfies $$\hat{b}(v_t)=b(v_t)+\alpha^{2}k^{L}_tv^{2\gamma}_t+ v^{\gamma}_t(\alpha H_t\rho+\alpha J_t)$$
 
 Define the random times $$\tau^{k}= \inf \{t: |\alpha^{2} k^{L}_t| \textless \gep^{1}\}$$   $$\tau^{J,H}=\inf \{t: |\alpha H_t\rho+\alpha J_t |\textgreater \gep^{2}\}.$$ Define the stopping time $\tau$ to be 
  \begin{equation}\label{tau'}
  \tau=(\tau^{k} \wedge \tau^{J,H})
  \end{equation}  
 
Proceeding, we have, on the stochastic interval $[0,\tau],$ the following lower bound on our drift: $$\hat{b}(v_t) \geq b(v_t)+\min(\gep^{(1)}, \gep^{(2)}) v^{2\gamma}_t-\max(\gep^{(1)}, \gep^{(2)})v^{\gamma}_t$$
  
 Let us recall the conditions of Lions and Musiela~\cite{L-M} on the coefficients and parameters of this system of stochastic differential equations
 \begin{eqnarray*} dS_t&=&S_t^{\beta} v_t^{\delta}dB_t\\
  dv_t&=&\alpha v_t^{\gamma}dW_t +b(v_t)dt  
    \end{eqnarray*}
such that $S$ is a martingale:
$\rho \textgreater 0,$ $\gamma + \delta \textgreater 1$ and \\
 $$ \limsup_{x \to +\infty} \frac{\rho\alpha x^{\gamma+\delta}+b(x)}{x} \textless \infty $$ 

Let us also recall the conditions on the coefficients and parameters of this system such that the process $S$ is a \textit{strict} local martingale: \\ 
$\rho \textgreater 0,$ $\gamma + \delta \textgreater 1$ and there exists $\phi(x),$ an increasing, positive, smooth function that satisfies
 \begin{equation*}   \int_{a}^{\infty} \frac{1}{\phi(x)}dx \textless \infty, \end{equation*} 
 where $a$ is some positive constant, and
 $$  \liminf_{x \to +\infty} \frac{\rho\alpha x^{\gamma+\delta}+b(x)}{\phi(x)} \textgreater 0$$

Our discussion has given rise to the following theorem:

\begin{theorem}\label{t2}
Assume the Standing Assumptions given in Hypothesis~\ref{SA}. Let $L$ be a random variable with a density. Assume that the following conditions are satisfied: 
 $$ \limsup_{x \to +\infty} \frac{\rho \alpha x^{\gamma+\delta}+b(x)}{x} \textless \infty $$ 
and $$  \liminf_{x \to +\infty} \frac{\rho\alpha x^{\gamma+\delta}+b(x) +\min(\gep^{(1)}, \gep^{(2)})x^{2\gamma}  -\max(\gep^{(1)}, \gep^{(2)})x^{\gamma}}{\phi(x)} \textgreater 0$$ 
Let $W$ and $B$ be correlated Brownian motions with correlation $\rho.$ Assume that  $\rho \textgreater 0$ and that $\gamma + \delta \textgreater 1.$ Let the process $S$ be the unique strong solution of the SDE 
 \begin{eqnarray}
dS_t&=&S_t^{\beta} v_t^{\delta}dB_t\\
  dv_t&=&\alpha v_t^{\gamma}dW_t +b(v_t)dt  
    \end{eqnarray}
   on $( P,\mbf)$. The solution $S$ is also the solution of
\begin{eqnarray} 
dS_t&=&S_t^{\beta} v_t^{\delta}dB_t\\
dv_t&=&b(v_s)dt +\alpha^{2}k^{L}_tv^{2\gamma}_tdt+ (\alpha v^{\gamma}_tH_t\rho+\alpha J_tv^{\gamma}_t)dt \end{eqnarray} 
on $(Q, \mbg).$
 
Then $S$ is positive, is a $( P,\mbf)$ martingale and a $(Q, \mbg)$ \textit{strict} local martingale on the stochastic interval $[0,\tau]$, where $\tau$ is given in~\eqref{tau'}. More specifically, we have $E[S^{\tau}_t] \textless S_0.$
  In the above, $\phi$ is an increasing, positive, smooth function that satisfies 
 \begin{equation*}  
  \int_{a}^{\infty} \frac{1}{\phi(x)}dx \textless \infty 
  \end{equation*} where $a$ is some positive constant. 
 \end{theorem}
 
 \begin{proof}
 
That $S$ is positive follows from Lemma~\ref{ell1}.  Defining the sequence of stopping times $T_n$ = $\inf \{t: v_t \geq n\},$ we have that the stopped process $S_{t \wedge \tau \wedge T_n}$ is a martingale.  The stopping time $T_\infty$ is the explosion time of $v.$  Therefore, we may write $$S_0=E[S_{t \wedge \tau \wedge T_n}]= E[S_{t\wedge \tau} 1_{\{t \textless T_n\}}]+ E[S_{T_n} 1_{\{T_n \leq t \wedge \tau\}}].$$  Since $E[S_{t \wedge \tau} 1_{\{t \textless T_n\}}]$ increases to $E[S_{t \wedge \tau}]$ as $n \rightarrow \infty$, we would have that $E[S_{t \wedge \tau}] \textless S_0$ for all $t$ if we can show that  $\liminf_{n \to +\infty} E[S_{T_n} 1_{\{T_n \leq t \wedge \tau\}}] \textgreater 0.$ 
   
   We have: $ E[S_{T_n} 1_{\{T_n \leq t \wedge \tau\}}] =\hat{P}(T_n \leq t \wedge \tau)$ where under $\hat{P},$ $v$ solves $$dv_t= \alpha v_t^{\gamma}dW_t +b(v_t)dt+\alpha^{2}k^{L}_tv^{2\gamma}_tdt+ v^{\gamma}_t(\alpha H_t\rho+\alpha J_t)dt+\rho{ v^{\gamma+\delta}_t}dt$$
   
   Now the condition $$  \liminf_{x \to +\infty} \frac{\rho\alpha x^{\gamma+\delta}+b(x) +\min(\gep^{(1)}, \gep^{(2)}) x^{2\gamma} -\max(\gep^{(1)}, \gep^{(2)})x^{\gamma}}{\phi(x)} \textgreater 0$$ is sufficient to guarantee that the explosion time of the SDE $$dv_t=\alpha v^{\gamma}_tdW_t+b(v_t)dt+\alpha^{2}+\min(\gep^{(1)}, \gep^{(2)}) v^{2\gamma}_t dt-\max(\gep^{(1)}, \gep^{(2)})v^{\gamma}_t dt +\rho{ v^{\gamma+\delta}_t}$$ can be made as small as we wish. 
   
   Thus we can invoke the comparison lemma and conclude that that the explosion time of the solution of the SDE $$dv_t=b(v_t)dt+\alpha^{2}k^{L}_tv^{2\gamma}_tdt+ v^{\gamma}_t(\alpha H_t\rho+\alpha J_t)dt+\rho{ v^{\gamma+\delta}_t}dt $$ can be made as small as possible. This means that we have, for any $t \textgreater 0,$ we have $$\liminf_{n \to +\infty} \hat{P}(T_n \leq t \wedge \tau) \textgreater 0.$$ This implies that for all $t$ we have $$E[S^{\tau}_{t}]=\textless S_0,$$ implying that $S_t$ is a local martingale that is not a martingale, and hence a strict local martingale.

\end{proof}

\begin{remark}
If we assume that there exists an $\epsilon \textgreater 0$ such that $\gamma \geq \frac{1+\epsilon}{2},$ we can use $\phi(x)=x^{1+\epsilon},$ and one can easily check that the following forms of $b(x)$ satisfy

$$ \limsup_{x \to +\infty} \frac{\rho \alpha x^{\gamma+\delta}+b(x)}{x} \textless \infty $$ and $$  \liminf_{x \to +\infty} \frac{\rho\alpha x^{\gamma+\delta}+b(x) +\min(\gep^{(1)}, \gep^{(2)}) x^{2\gamma} -\max(\gep^{(1)}, \gep^{(2)}) x^{\gamma}}{\phi(x)} \textgreater 0:$$ 

\begin{eqnarray*}
b(x)&=&K\ln(x)-\rho \alpha x^{\gamma+\delta}\\
b(x)&=&K\sin(x)-\rho \alpha x^{\gamma+\delta}\\
b(x)&=&Ke^{-ax}-\rho \alpha x^{\gamma+\delta}\\
b(x)&=&Kx^{m}-\rho \alpha x^{\gamma+\delta}
\end{eqnarray*}

In the above, $K$ and $a$ are positive constants, and $m$ is a constant satisfying $m \leq 1.$\\
\end{remark}

Before we continue, we must ensure that the sub-probability measure $Q$ defined above is a \textit{true} probability measure. Let us begin by defining the sequence of probability measures $Q_m$ by $$dQ_m=Z_{T\wedge T_m}dP$$ where $T_m=  \inf\{t: \int_{0}^{t} (H^{2}_s+J^{2}_s+2\rho J_sH_s)ds \ge h(m)\},$ for some function $h$. We then have
  
  $$E[e^{\frac{1}{2}\int_{0}^{t \wedge T_m}(H^{2}_s+J^{2}_s+2 \rho J_sH_s)ds}] \leq e^{\frac{1}{2}h(m)} \textless \infty.$$ In this case $H$ and $J$ must satisfy $$ k^{L}_tS^{2}_tv^{2\delta}_t+H_tS_tv^{2\delta}_t+\rho J_tS_tv^{\delta}=0$$ for all  $t \geq 0$ since we have assumed $Q$ to be a local martingale measure for $S.$

And so, we have $Q_m \ll P$ on $[0,T_m]$ for each $m,$ as well as that the $Q_m$ are true probability measures, since $Z^{T_m}_t$ is a true $\mathbb{G}$ martingale. 
  
Note that if $\{Z_{T \wedge T_m}\}_m$ is a uniformly integrable martingale, then $Q$ is equivalent to $P$ on $[0,T].$ This is because the uniform integrability of $(Z^{T_{m}})_{m}$ ensures the $L^{1}$ convergence of $Z^{T_m}$, i.e. $$\lim_{m \to \infty}E[Z_{t \wedge T_m}]=E[Z_{t \wedge \widetilde{T}}]=1,$$  where $\widetilde{T}=\lim_{m \to \infty}T_m.$ It is assumed that $\widetilde{T} \geq T.$ So we obtain that, for all $t$ in the interval $[0,T]:$ $E[Z_{t \wedge \widetilde{T}}]=E[Z_t]=1.$ Thus, $Q$ is equivalent to $P$ on $[0,T].$

 \section{The Discontinuous Case}\label{s3}
 
Let us now turn to the discontinuous case. That is, we assume that $S$ and $v$ follow SDEs of the form: 
\begin{eqnarray}\label{disceq}
dS_t&=&S_{t-} v_t^{\alpha}dM_t\\
  dv_t&=&\mu(v_t)dB_t +b(v_t)dt  
    \end{eqnarray}
    
We will assume that $\mu$ and $b$ are $C^{\infty}$ functions on $[0,\infty)$ and that $\mu$ is Lipschitz continuous on $[0,\infty)$ such that: 
\begin{eqnarray*}
\mu(0)&=&0\\
b(0) &\geq& 0\\
 \mu(x) \textgreater 0&\text{ if }&x>0\text{ and }\mu(x)=x\tilde{\mu}(x)\\
 b(x)&\leq& C(1+x)\text{ and }b(x)=x\tilde{b}(x)
 \end{eqnarray*}     
Note that the assumptions that $\mu$ and $b$ factor as $\mu(x)=x\tilde{\mu}(x)$ and $b(x)=x\tilde{b}(x)$ ensures a positive solution of the equation for $v$ in~\eqref{disceq}, even though it seems always true; but it is not, since we also require $\mu$ to be Lipschitz, and even if $\tilde{\mu}$ is Lipschitz, the function $x\mu(x)$ need be only locally Lipschitz. 

We assume $\alpha$ to be positive. In the above, $B$ is a standard Brownian motion and $M$ is a discontinuous martingale such that $[M,M]$ is locally in $L^{1}$ and such that $d \langle M,M \rangle_t = \lambda_t dt.$  Let us note that the conditions imposed on the coefficients $b$ and $\mu$ of the volatility are sufficient to ensure the existence and uniqueness of a nonnegative solution $v_t$ such that $E[\sup_{t \in [0,T]}|v^{p}_t|] \textless \infty$ for $1 \leq p \leq \infty.$ Last we assume that the processes $v$ and $M$ satisfy: 
\bee\label{16e1}
\Delta( \int_{0}^{t}v^{\alpha}_sdM_s) \textgreater -1,
\eee
 i.e., for all  $t$, ${v_t}_{-}^{\alpha} \Delta(M_t) \textgreater -1.$ (We are using the standard notation that for a c\`adl\`ag process $X$ that $\Delta X_t=X_t-X_{t-}$, the jump of $X$ at time $t$.) The above condition~\eqref{16e1} ensures that $S$ remains positive for all $t \geq 0.$ 

Let us proceed to expand the filtration $\mbf$ to obtain $\mbg$ by an initial expansion, and compute the canonical expansion of $S$ under $(P,\mbg).$ We obtain the canonical decomposition of the process $S$ under $\mbg$ via the theory of Jacod~\cite{JAC}. (The reader can consult~\cite[Chapter VI]{PP} for a pedagogic treatment of the subject.) Jacod proves the existence of an $\mbf$ predictable process $k^{L}_t$ such that
 \begin{equation*} \langle q^{L},S \rangle=k^{L}q^{L}_{-} \cdot \langle S,S\rangle  \end{equation*}
The function $k^{L}_t$ satisfies $k^{L}_t=\frac{h^{L}}{q^{L}}$ if $q^{L} \textgreater 0$ and $k^{L}_t=0$ otherwise. In the above, $h^{L}_t$ is the density process such that we have \begin{equation*} d \langle q^{L}_t,S \rangle_t=h^{L}_td\langle S,S \rangle_t \end{equation*}
Jacod's theorem also tells us that the following process is a $\mbg$ local martingale: 
\begin{equation*} \widetilde{S_t}=S_t-\int_{0}^{t}k^{L}_td {\langle S,S \rangle}_t \end{equation*}

We obtain: 
    
\begin{eqnarray*}    
S_t&=&S_0+\int_{0}^{t}S_{s-}v^{\alpha}_sdM_s- \int_{0}^{t}k^{L}_sS^{2}_sv^{2\alpha}_s\lambda_sds+ \int_{0}^{t}k^{L}_sS^{2}_sv^{2\alpha}_s\lambda_sds\\
v_t&=&v_0+\int_{0}^{t}\mu(v_s)dB_s -\int_{0}^{t}k^{L}_s\mu(v_s)^{2}ds +\int_{0}^{t}b(v_s)ds+\int_{0}^{t}k^{L}_s\mu(v_s)^{2}ds
\end{eqnarray*}

Here $\int_{0}^{t}S_{s-}v^{\alpha}_sdM_s- \int_{0}^{t}k^{L}_sS^{2}_sv^{2\alpha}_s\lambda_sds$ and $\int_{0}^{t}\mu(v_s)dB_s -\int_{0}^{t}k^{L}_s\mu(v_s)^{2}ds$ are $(P,\mbg)$ local martingales, and $ \int_{0}^{t}k^{L}_sS^{2}_sv^{2\alpha}_s\lambda_sds$ and $\int_{0}^{t}b(v_s)ds+\int_{0}^{t}k^{L}_s\mu(v_s)^{2}ds$ are finite variation processes.

We perform a Girsanov transform, to switch to a probability measure $Q$ which is equivalent to $P$, under which $S$ is a local martingale. We can do this as long as we assume the condition~\eqref{16e2}, given in Theorem~\ref{t7} (below). As in the previous cases, let $Z_t=E[\frac{dQ}{dP}|\mathcal{G}_t].$ Writing 
$$Z_t=1+ZH\cdot B_t +ZJ \cdot M_t,$$ 
where $\cdot$ represents stochastic integration, for $\mathbb{G}$ predictable processes $J_t$ and $H_t$, we have the following decompositions for $S$ and $v$ under $(Q, \mbg):$
\begin{eqnarray*} 
S_t&=&S_0+\int_{0}^{t}S_{s-}v^{\alpha}_sdM_s-\int_{0}^{t}k^{L}_sS^{2}_sv^{2\alpha}_s\lambda_sds -\int_{0}^{t}(\lambda_s H_sS_sv_s^{\alpha}+\rho J_sS_sv_s^{\alpha})ds\\  &&+\int_{0}^{t}k^{L}_sS^{2}_sv^{2\alpha}_s\lambda_sds +\int_{0}^{t}(\lambda_sH_sS_sv_s^{\alpha}+\rho J_sS_sv^{\alpha}_s)ds \end{eqnarray*}

\begin{eqnarray*}
v_t&=&v_0+\int_{0}^{t}\mu(v_s)dB_s -\int_{0}^{t}k^{L}_s\mu(v_s)^{2}ds -\int_{0}^{t}(H_s\rho \mu(v_s)+J_s\mu(v_s))ds\\ 
&&+\int_{0}^{t}b(v_s)ds+\int_{0}^{t}(H_s\rho \mu(v_s)+J_s\mu(v_s))ds +\int_{0}^{t}k^{L}_s\mu(v_s)^{2}ds \end{eqnarray*}

Since we have assumed that under $(Q,\mbg),$ $S$ is a local martingale, we set the finite variation term in its decomposition to zero: 

\bee\label{nne10}\lambda_tk^{L}_tS^{2}_tv^{2\alpha}_t+\rho J_tS_tv^{\alpha}_t+\lambda_tH_tS_tv^{\alpha}_t=0 \eee

Our new drift, which we will call $\hat{b}(v_t)$, under $(Q, \mbg)$ is given by the following:  \bee\label{nne11}\hat{b}(v_t)=b(v_t)+H_t\rho \mu(v_t)+J_t\mu(v_t)+k^{L}_t\mu(v_t)^{2}.\eee

$S$ and $v$ now solve, under $(Q, \mathbb{G})$ 

\begin{eqnarray*}
dS_t&=&S_{t-}v^{\alpha}_tdM_t+k^{L}_tS^{2}_tv^{2\alpha}_t\lambda_tdt +(\lambda_tH_tS_tv_t^{\alpha}+\rho J_tS_tv^{\alpha}_t)dt \\
dv_t&=&\mu(v_t)dB_t +b(v_t)dt+k^{L}_t\mu(v_t)^{2}dt+(H_t\rho \mu(v_t)+J_t\mu(v_t))dt 
\end{eqnarray*}

Let us remark that in this case, the following relation holds: \begin{equation*} \langle q^{L}, S \rangle =\int_{0}^{t}k^{L}_sq^{L}_sS^{2}_sv^{2\alpha}_s\lambda_sds  \end{equation*} Returning to the decomposition of the volatility we just arrived at, we again note that we can no longer represent the drift in deterministic terms as simply functions of the real variable $x$, so we cannot immediately invoke the results of Lions \& Musiela. To address this, let us fix $0 \textless \gep^{(1)} \textless k^{L}_0$ and $ |\rho H_0+J_0| \textless \gep^{(2)}$ and define the following random times: 
  
 \begin{eqnarray*}
\tau^{k}&=& \inf \{t: |k^{L}_t| \textless \gep^{(1)}\}\\   
\tau^{H,J}&=&\inf \{t: |\rho H_t+J_t| \textgreater \gep^{(2)}\}
\end{eqnarray*} 

Now define the stopping time $\tau$ to be 
  \begin{equation}\label{tau_d}
  \tau=(\tau^{k}\wedge \tau^{H,J}).
  \end{equation}

    On the stochastic interval $[0,\tau],$ we have the following lower bound on our drift coefficients: 
\begin{equation*}
\hat{b}(v_t)= b(v_t)+k^{L}_t\mu^{2}(v_t)+(\rho H_t+J_t)\mu(v_t) \geq b(v_t)+\min(\gep^{(1)}, \gep^{(2)})\mu^{2}(v_t)-\max(\gep^{(1)}, \gep^{(2)})\mu(v_t)
\end{equation*}
 
 The above discussion gives us the following result. 

From this discussion, we have arrived at the following theorem:

\begin{theorem}\label{t7}
Let $S_t$ be the strong solution under $(P, \mathbb{F})$ of 
\begin{eqnarray}\label{s_d} 
dS_t&=&S_{t-} v_t^{\alpha}dM_t\\
 dv_t&=&\mu(v_t)dB_t +b(v_t)dt  
 \end{eqnarray}
$S$ is also the solution, under $(Q, \mathbb{G})$ of \begin{eqnarray*}
dS_t&=&S_{t-}v^{\alpha}_tdM_t\\
dv_t&=&\mu(v_t)dB_t +b(v_t)dt+k^{L}_t\mu(v_t)^{2}dt+(H_t\rho \mu(v_t)+J_t\mu(v_t))dt 
\end{eqnarray*}
Assume:
\bee\label{16e2}
E[\e^{{\int_{0}^{T}v^{2\alpha}_sd\langle M^{d}, M^{d} \rangle}_s +\frac{1}{2}\int_{0}^{T}v^{2 \alpha}_sd\langle M^{c},M^{c} \rangle_s }]\textless \infty.
\eee 
Assume also that  
\begin{equation*} \liminf_{x \to +\infty}(\rho\ x\mu(x)+b(x)+\min(\gep^{(1)}, \gep^{(2)}) \mu^{2}(x)-\max(\gep^{(1)}, \gep^{(2)}) \mu(x))\phi(x)^{-1} \textgreater 0\ 
\end{equation*} Then, the process $S$ is a true $(P, \mathbb{F})$ martingale and a $(Q, \mathbb{G})$ strict local martingale. Specifically, we have $E[S^{\tau}_t] \textless S_0$ where $\tau$ is given by~\eqref{tau_d}.
\end{theorem}

\begin{proof} [Proof of Theorem~\ref{t7}]

First note that the strong assumption given in~\eqref{16e1} ensures that $S_{-}$ is positive. From~\cite{PhilipKazu}, a sufficient condition for the solution $S$ of $dS_t=S_tv_t^{\alpha}dM_t.$ be a martingale on $[0,T]$ is that $E[\e^{{\int_{0}^{T}v^{2\alpha}_sd\langle M^{d}, M^{d} \rangle}_s +\frac{1}{2}\int_{0}^{T}v^{2 \alpha}_sd\langle M^{c},M^{c} \rangle_s }]\textless \infty.$ (In Remark~\ref{16r1} following this proof we present an alternative condition.)
 
Let us now display sufficient conditions for the solution $S$ of~\eqref{s_d} under $(Q, \mathbb{G})$ to be a strict local martingale.

Define a sequence of stopping times $T_n$ by $\inf \{t: v_t \geq n\},$. We have that the stopped process $S_{t \wedge \tau \wedge T_n}$ is a martingale.  The stopping time $T_\infty$ is the explosion time of $v.$ Therefore, we may write $$S_0=E[S_{t \wedge \tau \wedge T_n}]= E[S_t 1_{\{t \wedge \tau \textless T_n\}}]+ E[S_{T_n} 1_{\{T_n \leq t \wedge \tau\}}].$$ As we saw in the continuous case, since $E[S_{t \wedge \tau} 1_{\{t \wedge \tau \textless T_n\}}]$ increases to $E[S_{t \wedge\tau}]$, we would have that $E[S_{t \wedge \tau}]  \textless S_0$ for all $t$ if we can show that  $\liminf_{n \to +\infty} E[S_{T_n} 1_{\{T_{n} \leq t \wedge \tau\}}] \textgreater 0.$ 

Now $E[S_{T_n} 1_{\{T_{n} \leq t \wedge \tau\}}] \textgreater 0= \hat{P}(T_{n} \leq t \wedge \tau),$ where under $\hat{P},$ $v$ solves
 \begin{equation*} dv_t=\mu(v_t)dB_t +b(v_t)dt+k^{L}_t\mu^{2}(v_t)+(\rho H_t+J_t)\mu(v_t) +\rho v_t\mu(v_t)dt \end{equation*}

The condition  \begin{equation*} \liminf_{x \to +\infty}(\rho\ x\mu(x)+b(x)+\min(\gep^{(1)}, \gep^{(2)})\mu^{2}(x)-\max(\gep^{(1)}, \gep^{(2)})\mu(x))\phi(x)^{-1} \textgreater 0\ \end{equation*} 
 is sufficient to guarantee that the explosion time of the stochastic differential equation  
\begin{equation*}
  dv_t=\mu(v_t)dW_t +b(v_t)dt + \rho v_t \mu(v_t)+\min(\gep^{(1)}, \gep^{(2)}) \mu^{2}(v_t)dt-\max(\gep^{(1)}, \gep^{(2)})\mu(v_t)dt
    \end{equation*}
 can be made as small (in an appropriate sense) as we wish.
 
The comparison theorem implies that the solution to the SDE $$dv_t=\mu(v_t)dW_t+b(v_t)dt+ \rho \mu(v_t)v_tdt+ k^{L}_t\mu^{2}(v_t)+(\rho H_t+J_t)\mu(v_t)dt$$ is $Q$ almost surely greater than or equal to that of the SDE  $$dv_t=\mu(v_t)dW_t +b(v_t)dt +\min(\gep^{(1)}, \gep^{(2)})\mu^{2}(v_t)dt -\max(\gep^{(1)}, \gep^{(2)})\mu(v_t)dt$$ for all $t \in [0, \tau].$ Thus, since the explosion time of the SDE  
\begin{eqnarray*} 
dS_t&=&S_t v_tdB_t\\
dv_t&=&\mu(v_t)dW_t +b(v_t)dt +\min(\gep^{(1)}, \gep^{(2)})(v_t)dt-\max(\gep^{(1)}, \gep^{(2)})\mu(v_t)dt 
\end{eqnarray*}
can be made as small as possible, the explosion time $T_\infty,$ of 
\begin{eqnarray*}
dS_t&=&S_t v_tdB_t\\
dv_t&=&\mu(v_t)dW_t +b(v_t)dt +\min(\gep^{(1)}, \gep^{(2)})(v_t)dt+\max(\gep^{(1)}, \gep^{(2)})\mu(v_t)dt 
 \end{eqnarray*}
 can be made as small as possible as well. 
 
This means that, for all $t$, we have $\hat{P}(T_\infty \leq t \wedge \tau) \textgreater 0$.  This implies that for all $t$ we have $$E[S^{\tau}_{t}]=\textless S_0,$$ implying that $S_t$ is a local martingale that is not a martingale, and hence a strict local martingale.
\end{proof}

\begin{corollary}
Let $M$ be a L\'{e}vy martingale. Then, by the L\'{e}vy-It\^{o} decomposition, $$M_t=W_t+\int_{|x| \textless 1}x(N(;[0,t],dx)-t\nu(dx)+ \sum_{0 \textless s \textless t} \Delta M_s1_{\{|\Delta M_s| \geq 1\}}- \alpha t$$ In the above, $N_{t}(\Lambda)$ is a Poisson random measure, $\alpha t= E[\sum_{0 \textless s \textless t} \Delta M_s1_{|\Delta M_s| \geq 1}]$  and $\nu(dx)$ is the L\'evy measure of the process $M_t:$ $$\nu(\Lambda)= E[N^{1}(\Lambda)].$$ $M$ satisfies: $d\langle M,M\rangle_t=(1+\int_{\mathbb{R}}x^{2} \nu(dx))dt=cdt.$
Assume that $E[\e^{\int_{0}^{T} (\frac{1}{2}+\int_{\mathbb{R}} x^{2} \nu(dx))v^{2\alpha}_sds}] \textless \infty.$ This is satisfied if $\int v^{\alpha}_sdM_s$ is locally square integrable. Assume also that~\eqref{16e1} holds and that\begin{equation*} \liminf_{x \to +\infty}(\rho\ x\mu(x)+b(x)+\min(\gep^{(1)}, \gep^{(2)}) \mu^{2}(x)-\max(\gep^{(1)}, \gep^{(2)}) \mu(x))\phi(x)^{-1} \textgreater 0\ \end{equation*} Then, the process $S$ of~\eqref{disceq} is a true $(P, \mathbb{F})$ martingale and a $(Q, \mathbb{G})$ strict local martingale. 

\end{corollary}

\begin{remark}\label{16r1}
We now give an alternative way to ensure that when we change probabilities from $P$ to $Q$ after a filtration enlargement, that $Q$ is indeed a true probability measure and not a sub probability measure. This is an alternative to assuming that the continuous paths equivalent of~\eqref{16e2} holds, although it is related. Let us ensure that, in the discontinuous case we have just encountered, the subprobability measure $Q$ we defined is a \textit{true} probability measure. We will begin by defining the sequence of probability measures $Q_m$ by $$dQ_m=Z_{T\wedge T_m}dP$$ where $Z$ is the Dol\'{e}ans-Dade exponential of $(\int_{0}^{t}H_sdB_s+\int_{0}^{t}J_sdM_s)$ and $T_m=  \inf\{t: \int_{0}^{t} (H^{2}_s+J^{2}_s+2\rho J_sH_s+J^{2}_s\int_{\mathbb{R}} x^{2} \nu(dx))ds \ge h(m)\},$ for some function $h$. We then have $$E[e^{\frac{1}{2}\int_{0}^{t \wedge T_m}(H^{2}_s+J^{2}_s+2\rho J_sH_s+H^{2}_s(\int_{\mathbb{R}} x^{2} \nu(dx)))ds}] \leq e^{\frac{1}{2}h(m)} \textless \infty.$$ Recall that the relation $$k^{L}_tS^{2}_tv^{2\alpha}_tc+\rho J_tS_tv^{\alpha}_t+cH_tS_tv^{\alpha}_t=0 $$ holds true for all $t \geq 0.$ 
   
Continuing, we have $Q_m \ll P$ on $[0,T_m]$ for each $m,$ as well as that the $Q_m$ are true probability measures, since $Z_{t}^{T_m}$ is a true $\mathbb{G}$ martingale. 
  
Note that if $\{Z_{T \wedge T_m}\}_m$ is a uniformly integrable martingale, then $Q$ is equivalent to $P$ on $[0,T].$ This is because the uniform integrability of $(Z^{T_{m}})_{m}$ ensures the $L^{1}$ convergence of $Z^{T_m}$, i.e. $$\lim_{m \to \infty}E[Z_{t \wedge T_m}]=E[Z_{t \wedge \widetilde{T}}]=1,$$  where $\widetilde{T}=\lim_{m \to \infty}T_m.$ It is assumed that $\widetilde{T} \geq T.$ So we obtain that, for all $t$ in the interval $[0,T]:$ $E[Z_{t \wedge \widetilde{T}}]=E[Z_t]=1.$ Thus, $Q$ is equivalent to $P$ on $[0,T].$ 

We take this opportunity to mention that this idea (discovered independently by the first author) is developed in a beautiful (and more general) way in the recent paper of J. Blanchet and J. Ruf~\cite{BR}.

\end{remark}
\subsection{Examples}\label{ss1}

Let us now consider some examples. 

\begin{example}[Mansuy and Yor~\cite{MansuyYor}]
 $S$ and $v$ solve
 \begin{eqnarray*} dS_t&=&S_t v_tdB_t; \qquad S_0=1\\
  dv_t&=&\mu(v_t)dW_t +b(v_t)dt;  \quad v_0=1
 \end{eqnarray*} and $L=B_T.$ In this case, we have \begin{equation*} k^{L}_t= S_tv_t\frac{B_T-B_t}{T-t} \end{equation*} We have $k^{L}_0=S_0v_0 \frac{B_T}{T},$ and indeed, $Q(\omega:k^{L}_0)\textgreater 0.$ Here, It is immediately apparent that the process $k$ has right-continous paths.
\end{example}

\begin{example}[Mansuy and Yor~\cite{MansuyYor}]
 $S$ and $v$ solve
 \begin{eqnarray*} dS_t&=&S_t v_tdB_t; \qquad S_0=1\\
  dv_t&=&\mu(v_t)dW_t +b(v_t)dt;  \quad v_0=1
 \end{eqnarray*} and $L=T_a,$ the first hitting time of $a$ of the Brownian motion $B_t.$ In this case, we have $k^{L}_t=-\frac{1}{a-B_t} +\frac{a-B_t}{T_a-t}.$ Again, It is immediately apparent that the process $k$ has right-continous paths and that $Q(\omega:k^{L}_0)\textgreater 0.$
 
 \end{example}
 
 \begin{example}[The Countable Partition Case]
 
 Let $S$ and $v$ solve
 \begin{eqnarray*} dS_t&=&S_t v_tdB_t; \qquad S_0=1\\
  dv_t&=&\mu(v_t)dW_t +b(v_t)dt;  \quad v_0=1
 \end{eqnarray*}   

 Let us assume that we have a countable partition of the sample space such that $A_i \cap A_j= \emptyset$ if $i \neq j$ and $\bigcup\limits_{i=1}^n A_k=\Omega$ and that the information encoded in $L$ can be modeled as $L= \sum\limits_{i=1}^n a_i1_{A_i}.$  Note that the vector process $\begin{bmatrix}S_t \\ v_t \\ \end{bmatrix}$ is a strong Markov process. Let us define our partition in terms of this Markov process. Fix a time $T \textgreater 0$ and assume we have half open sets $(\alpha_{i}, \beta_{i}]$ such that $\bigcup\limits_{i=1}^n (\alpha_{i}, \beta_{i}]= \mathbb{R}$ and $(\alpha_{i}, \beta_{i}] \cap (\alpha_{j}, \beta_{j}]  =\phi,$ $i \neq j.$ Let $A_{i}=\{\omega: S_T \in (\alpha_{i}, \beta_{i}]\}$ 
 
   If, in this case, we have $a_i \textgreater 0$ and $P(A_i) \textgreater 0$ then we have that the process $k$ satisfies $k^{L}_0 \textgreater 0.$ Consider the sequence of martingales $N^{i}_t=E[1_{A_{i}} |\mathcal{F}_t].$ By the Kunita-Watanabe inequality, there exists processes $\xi^{i}_t$ such that $d[N^{i},S]_t=\xi^{i}_td[S,S]_t.$  Now the determination of whether or not $k$ has right-continuous paths is tantamount to the determination of whether or not, for each $i,$ the processes $\xi^{i}_t$ possess right-continuous paths. Since we are in the Brownian framework, we can employ martingale representation to write: $N^{i}_t=\int_{0}^{t}h^{i}_sdB_s+\int_{0}^{t}g^{i}_sdW_s$ 
   
   Then, $[N^{i},S]_t=\int_{0}^{t}(h^{i}_s+\rho g^{i}_s)S_sv_sds,$ $\rho$ being the correlation of the Brownian motions $B$ and $W,$ and $\xi^{i}_t$ is such that $(h^{i}_t+\rho g^{i}_t)dt=\xi^{i}_tS_tv_tdt.$ If we can prove then, that for each $i$, the processes $h^{i}_t+\rho g^{i}_t$ possess right-continuous paths, then we are done. 
   
   We now apply the results of  ~\cite{JSP}, specifically corollary $2.5$  For all $i$, we can write $f^{i}(\begin{bmatrix}S_T\\ v_T \\ \end{bmatrix}) = \begin{bmatrix}1_{\{ (\alpha_{i}, \beta_{i}]\}} (S_T)\\ 0 \end{bmatrix}.$  For each $i$, we need to find an approximating sequence of functions $f^{i,n}(x)$ such that  $f^{i,n}(\begin{bmatrix}S_T\\ v_T \\ \end{bmatrix})\rightarrow  f^{i}(\begin{bmatrix}S_T\\ v_T \\ \end{bmatrix})$ in $L^{2}(P).$  For all $i$, $f^{i,n}(x)$ must be Borel functions, and $(t,y) \rightarrow P_tf^{i,n}(y)$ on $(0, \infty)$ must be once differentiable in $t$ and twice differentiable $x,$ all partial derivatives being continuous. $P_t$ denotes the transition semigroup of the process  $\begin{bmatrix}S_t \\ v_t \\ \end{bmatrix}$ Note that this holds when the functions $f^{i,n}(x)$ are twice differentiable, with continuous second derivative, and with compact support.  Note that this differentiability is just what we need to apply Theorem 3.2 of~\cite{MPZ}, which gives us that the corresponding process $(h^{i,n}_s+\rho g^{i,n}_s)_{0\leq s\leq T}$ has c\`dal\`ag paths, for each $n$. 
   
  We have that for each $i$ an approximating sequence of functions $f^{i,n}(x)$ of $f^{i}(x)$ is given by $f^{i}(x) \ast \phi^{n}(x),$ where $\phi^{n}(x)$ is a sequence of mollifiers. For example, we can take $\phi^{n}(x)= n^{2}\phi(nx),$ where $\phi(x)=c\e^{-\frac{1}{1-||x||^{2}}}\chi_{[-1,1]}(x).$ We have that $f^{i}(x) \ast \phi^{n}(x)$ is smooth and with compact support. It converges uniformly and thus in $L^{2}$ to $f^{i}(x).$ Moreover, we also have the uniform (in $t$) convergence of $P_tf^{i,n}(y) \rightarrow P_tf^{i}(y).$ Now by Corollary $2.5$ from ~\cite{JSP}, we have for each $i$ the existence of an explicit representation of  a version of the process $h^{i}_t+\rho g^{i}_t$ which indeed possesses c\`{a}dl\`{a}g paths, since it is the uniform (in the time variable) limit  of the c\`dal\`ag processes $h^{i,n}_t+\rho g^{i,n}_t$.

\end{example}

\section{Connections to Mathematical Finance}\label{s4}
 The motivation for this work is to relate possible economic causes of financial bubbles to mathematical models of how they might arise, from within the martingale oriented absence of arbitrage framework. We use the economic cause of speculative pricing that comes from overexcitement of the market due to the disclosure of new information. Examples might be the announcement of a new medicine with major financial consequences (such as a ``cure'' for the common cold, to exaggerate a bit), a technological breakthrough (this is the thesis of John Kenneth Galbraith, for example~\cite{JKG}), a resolution of some sort of political instability, a weather event (such as an early frost for the Florida orange crop), etc. The obvious and intuitive manner to model such an event is by the addition of new observable events to the underlying filtration, and an established way to do that is via the theory of the ``expansion of filtrations." This theory was developed in the 1980s, and a recent presentation can be found, for example, in~\cite[Chapter VI]{PP}. 
 
 The theory of the expansion of filtrations and the martingale theory of an absence of arbitrage do not mesh well, as papers of Imkeller~\cite{PI}, Fontana et al~\cite{JB1}, and the PhD thesis of Anna Aksamit~\cite{ANA} have detailed. Many more references are provided in those papers. Therefore one has to be careful both as to how one expands the filtration as well as to what one means by an absence of arbitrage. Here we use the approach of an ``initial expansion,'' although we interpret it as occurring at a random (stopping) time. We work in an incomplete market setting where there are an infinite number of risk neutral measures; in particular we take a stochastic volatility framework. We show how the expansion of filtrations creates a drift even in a drift free model (this is well known) and then we need to change the risk neutral measure to remove the drift created by the addition of new information. The insight is that under this new risk neutral measure with the new enlarged filtration, the price process changes from a martingale to a strict local martingale. This has financial significance: It has been shown over the last decade that on compact time sets, a price process models a financial bubble if and only if it is  strict local martingale under the risk neutral measure; thus we have shown how a non bubble price process can become a bubble price process after the arrival of new information (via an expansion of the filtration). Our ideas were inspired by the previous works of Carlos Sin~\cite{C-S} and Biagini-F\"ollmer-Nedelcu~\cite{BFN} who were interested in bubble formation, but did not relate it to the expansion of filtrations. 
  
Finally, we remark that this is different from the modeling of insider information, another popular use of the expansion of filtrations; see for example~\cite{ANA,Roseline}.



\end{document}